\documentclass[10pt]{amsart}

\usepackage[T1]{fontenc} 

\usepackage{amssymb,amsmath,amsfonts}

\usepackage{tikz}
\usepackage{algorithm}
\usepackage{algpseudocode}
\usepackage{float}
\usepackage{hyperref}
\usepackage{verbatim}

\usepackage{enumitem}

\newtheorem{theorem}{Theorem}

\newtheorem{example}{Example}
\newcommand{\lcm}{ {\rm lcm}}

\makeatletter
\newtheorem{repeatthm@}{Theorem}
\newenvironment{repeatthm}[1]{%
    \def\therepeatthm@{\ref{#1}}
    \repeatthm@
}
{\endrepeatthm@}
\makeatother

\algnewcommand\algorithmicforeach{\textbf{for each}}
\algdef{S}[FOR]{ForEach}[1]{\algorithmicforeach\ #1\ \algorithmicdo}

\title{Algorithms for Carmichael numbers }

\author{Andrew Shallue}
\author{Jonathan Webster}
\address[1]{Illinois Wesleyan University}
\address[2]{Butler University}

\begin{document}

\begin{abstract}
Our primary concern is the computational complexity of algorithms that find all Carmichael numbers less than some specified bound $B$.    We have three related results.  
First, we show \textsc{carmichaels} is in \textbf{\textsc{p}}, where only the run-time is conditioned on the ERH.  
Second, we state a heuristically optimal tabulation algorithm,  
which is the first asymptotic improvement to tabulation algorithms in the $50$ years since Swift first described the prime-by-prime approach \cite{swift}.  
Third, we implemented a related algorithm that tabulated $100$ times further while only doing about $5$ times the work of the prior tabulation.
We found $308,279,939$ Carmichael numbers less than $10^{24}$ and we provide some statistics on these numbers.  
\end{abstract}

\maketitle

\section{Introduction}

\textit{Carmichael numbers} are the composite numbers $n$ satisfying $b^n \equiv b \pmod{n}$ for every integer $b$.  
Fermat's little theorem says that the congruence must hold whenever $n$ is a prime and so these numbers might also be called \textit{absolute Fermat pseudoprimes}.  
In \cite{RDcarm}, Robert Carmichael proved that such numbers must be square-free and satisfy $(p-1) | (n-1)$ for every prime $p|n$.  
He also gave four examples including $561$, the smallest Carmichael number\footnote{
In 1899, Korselt had previously proven this result but gave no examples.  
In 1885, \v{S}imerka found examples but did not state how they were found.
}.  
Given their status as pseudoprimes and the many open conjectures regarding their behavior, finding examples of Carmichael numbers and  tabulating them has been an ongoing project for the last $50$ years.  

In 1975, J.D. Swift described an algorithm used to find all Carmichael numbers up to $10^{9}$ \cite{swift}. 
Every tabulation thereafter followed the same basic approach.  
The improvements offered by Jaeschke \cite{jaeschke_carm}, Pinch \cite{pinch15}, and Shallue \& Webster \cite{sw_ants2, sw_ants3} all involved faster ways of handling a small set of inputs. 
Thus while the tabulation bound was pushed (eventually to $10^{22}$), the dominating asymptotic for all these approaches remained Kn\"odel's bound, generated algorithmically by constructing Carmichael numbers prime by prime.

\begin{theorem}[\cite{knodel}]\label{thm:input_count}
Let $C(B)$ be the count of Carmichael numbers less than $B$.  Then
\[ C(B) < B\exp\left\{\left(-\frac{1}{\sqrt{2}} +o(1)\right)\sqrt{\log B \log_2  B}\right\}. \]
Let $\mathcal{K}(B)$ denote this bound.
\end{theorem}
\noindent In the above, $\log_2 n$ denotes the twice-iterated natural logarithm, and $\log_k n$ will denote the $k$-fold iterated natural logarithm. 

Shortly after Kn\"odel proved his result, Erd\H{o}s found a better upper bound \cite{erdos_CN}.  
We state an improved version proven by Pomerance, Selfridge, and Wagstaff.

\begin{theorem}[{\cite[Theorem 6]{pom_self}}] \label{thm:CN_count}
For each $\epsilon > 0$, there is a $B_0(\epsilon)$ such that for all $B > B_0(\epsilon)$, we have 
\[C(B) \leq  B \exp \left\{ \frac{-(1-\epsilon)\log B \log_3 B}{ \log_2 B }\right\}.\]
Let $\mathcal{E}(B)$ denote this bound.   
\end{theorem}

In \cite{dist_psp, pom_self}, a heuristic argument is given that $C(B) \sim \mathcal{E}(B)$.  
However, the best proven lower bounds are of the form $B^\alpha$ for some fixed $\alpha$.  
Following the seminal work of \cite{inf_carm}, where the infinitude of Carmichael numbers was first proven with $\alpha = 2/7$, we have a series of refinements improving the exponent to $0.332$ \cite{harman1}, then $0.33336704$ \cite{harman2}, and finally $0.3389$ \cite{lichtman}.  
Those strictly guided by empirical observations may find the lower bounds more convincing than the heuristic argument because at our current tabulation range the largest exponent observed is $0.35$.  

The major result of the present work is an algorithm with $\mathcal{E}(B)$ as its asymptotic running time.
This supersedes all previous work on Carmichael number tabulation, giving the first asymptotic improvement since Swift first described his approach.

\begin{theorem}\label{thm:optimal}
For each $\epsilon > 0$, there is an $B_0(\epsilon)$ such that for all $B > B_0(\epsilon)$, there is an algorithm that can find all Carmichael numbers less than $B$ in at most $\mathcal{E}(B)$ 
queries of ``Is $n$ a Carmichael number?''.  
\end{theorem}

Thus the computational complexity of answering the query ``is $n$ a Carmichael number?'', which we denote as \textsc{\textsc{carmichaels}}, is of central importance to the present work.
This query is also of independent interest; it is a bit of folklore that Carmichael numbers are easy to factor based on the belief that most Carmichael numbers are smooth.
Since some Carmichael numbers are difficult to factor, factoring seems unavoidable, and  
general factorization algorithms are not in \textbf{\textsc{p}}, the following theorem is noteworthy.
% Where previous attempts required factoring, we prove the following.

\begin{theorem}\label{thm:iscarmichael}
\textsc{carmichaels} is in \textbf{\textsc{p}}, where correctness is unconditional and the run-time analysis is conditioned on the Extended Riemann Hypothesis (ERH). 
\end{theorem}

Theorems \ref{thm:optimal} and \ref{thm:iscarmichael} together justify the asymptotic running time of $\mathcal{E}(B)$.  
% Also this is not consistent with what we claim in Theorem 12 at the beginning of section 6. After discussion, we will modify theorem 12 and add a corollary there.
If the heuristic argument of \cite{dist_psp, pom_self} is correct, then we can further conclude that our algorithm is optimal. 

A pleasing feature of the present work is that the Fermat definition of Carmichael numbers plays a central role, where other tabulations ignore it in favor of Korselt's criterion.  
Swift wrote, 
``The computer programs used depended explicitly on congruence properties of the CN’s with respect to their component primes rather than on the pseudoprimality with respect to any particular base.'' 

In \cite[Theorem 8]{sw_ants3} we proved that $\mathcal{K}(B)$ is algorithmically realized.
The primary algorithmic advance in the present work is our use of Theorem \ref{thm:iscarmichael} to search for composite $R$ that completes a preproduct $P$ to form Carmichael numbers of the form $n = PR$.  
This provides an algorithmic realization of $\mathcal{E}(B)$.
Relaxing the requirement for completions to be prime is the source of the separation between $\mathcal{K}(B)$ and $\mathcal{E}(B)$.

In the next section we establish notations and state some elementary theorems used in tabulation algorithms. 
In Section \ref{sec_carinp} we state the central algorithm to prove the results on \textsc{carmichaels}.  
Since we claim that our new tabulation algorithms supersedes prior work,   
we use Section \ref{sec:survey} to provide a historical survey of the prime-by-prime algorithms.
The section concludes with a heuristic argument that they all had $\mathcal{K}(B)$ as the run-time.  
Section \ref{sec:faster1} gives two asymptotically faster algorithms.
The first of which is a practical algorithm that was implemented and the second is a heuristically optimal algorithm.
Section \ref{sec:the_tab} presents results of the computation and the tabulation of Carmichael numbers up to $10^{24}$.
We conclude by presenting some related problems.  

\section{Number theoretic results on Carmichael numbers}\label{sec:elementary}

We frequently recognize Carmichael numbers $n$ via their prime factorization.

\begin{theorem}[Korselt's Criterion]
A composite number $n$ is a Carmichael number if and only if $n$ is squarefree and $(p-1) \mid (n-1)$ for all prime divisors $p$ of $n$. 
\end{theorem}

Robert Carmichael proved Korselt's criterion in the context of the reduced totient function or the Carmichael function $\lambda(n)$, 
\begin{align*}
\lambda\left(p^{\alpha}\right) &= \begin{cases} \phi\left(p^{\alpha}\right) &\text{ if } \alpha \leq 2 \text{ or } p \geq 3 \\ \frac{1}{2} \phi\left(p^{\alpha}\right) &\text{ if } \alpha \geq 3 \text{ and } p = 2 \end{cases},
\end{align*}
\[ \lambda(n) = \lambda\left( {p_1}^{\alpha_1} \cdots {p_k}^{\alpha_k} \right) = \lcm \, \big(\lambda\left({p_1}^{\alpha_1}\right), \ldots, \lambda\left({p_k}^{\alpha_k}\right)\big), \]
where the $p_i$ are distinct prime divisors of $n$ and $\phi(n)$ is Euler's totient function.  We can now state Korselt's criterion in terms of $\lambda(n)$.

\begin{theorem}[Korselt's Criterion]
A composite number $n$ is a Carmichael number if and only if $\lambda(n) | (n-1)$. 
\end{theorem}

We reserve the following notation.
\begin{itemize}
\item  Let $P_k = \prod_{i = 1}^{k} p_i$ where $p_i < p_j$ iff $i < j$ and call $P_k$ a \textit{preproduct}.
\item  We drop the subscript $k$ from $P_k$ when the context does not require a specific count of primes.  
\item We let $p$ be the greatest prime factor of $P$.  
\item We will use $q$ or $r$ to denote other primes.  
\end{itemize}
   
\begin{theorem}[{\cite[Proposition 1]{pinch15}}] \label{thm:bounds_theorem}
    Let $n = P_d < B$ be a Carmichael number.
 
    \begin{enumerate}

        \item \label{first_bullet}Let $k < d$.  Then $p_{k+1} < (B/P_{k})^{1/(d-k)}$ and $p_{k+1} - 1$ is relatively prime to $p_i$ for all $ i \leq k$.  

        \item \label{second_bullet} $P_{d-1}p_d \equiv 1 \pmod{\lambda(P_{d-1})}$ and $p_d-1$ divides $P_{d-1} - 1$.

        \item \label{third_bullet} Each $p_i$ satisfies $p_i < \sqrt{n} < \sqrt{B}$.  
    \end{enumerate}
\end{theorem}

Part \ref{second_bullet} of Theorem \ref{thm:bounds_theorem} can be generalized: we have $ PR \equiv 1 \pmod{\lambda(P)}$ for arbitrary positive $R$ whenever $n$ is Carmichael and $n = PR$.  
 The quantity $P^{-1} \pmod{ \lambda(P)}$ is used frequently throughout and we let $r^{\star}$ denote the least positive integer in that residue class.  
 This might be considered an unconventional naming choice but it is due to the fact that we associate this quantity to $R$, the remaining factor to be constructed.  So, $R$ is in the arithmetic progression 
 \[\mathcal{R} = \{r^\star , r^\star + \lambda(P), r^\star + 2\lambda(P), \ldots, r^\star + k\lambda(P), \ldots \}. \]

 Products $P$ where $p_{j} -1$ is relatively prime to all $p_i$ for $i < j$ are called \textit{cyclic}.  If there exists a prime $q$ and a cyclic $P$ such that $Pq$ is also cyclic we say that $q$ is \textit{admissible} to $P$.

Part \ref{second_bullet} of Theorem \ref{thm:bounds_theorem} also implies that the number of Carmichael numbers with a fixed second largest prime factor is finite.  We have a stronger result that implies that the number of Carmichael numbers with a fixed third largest prime factor is also finite.  

\begin{theorem}  \label{thm:CD_theorem}
Let $n = Pqr$ with $p < q < r$ be a Carmichael number.  
            There are integers $2 \leq D < P < C$ such that, putting $\Delta = CD - P^2$, we have
            \[ q = \frac{(P-1)(P+D)}{\Delta} + 1, \quad  r = \frac{(P-1)(P+C)}{\Delta} + 1,\quad    P^2 < CD < P^2\left( \frac{p + 3}{p + 1} \right). \]     
\end{theorem}

A version of Theorem \ref{thm:CD_theorem} was proved by Beeger for $P$ prime \cite{beeger}, and 
Duparc generalized it to composite $P$ \cite{duparc}.  
Jaeschke seems to be the first to incorporate an algorithmic version \cite{jaeschke_carm}.  
The version above matches Pinch's use in \cite{pinch15}.

\section{\textsc{carmichaels} is in \textbf{\textsc{p}}, conditionally}\label{sec_carinp}

It had long been conjectured that \textsc{primes} was in \textbf{\textsc{p}} and the groundbreaking work of AKS established it \cite{AKS}.  
In a similar way, we address the computational complexity of \textsc{carmichaels}. 

The chief difficulty lies in answering ``yes'' to this query.   
We present the prior work on this problem, and show how answering ``yes'' is related to integer factorization. 
We give a brief review of Miller's primality test; it is motivational and necessary for our own approach. 
This leads to a simplistic Miller-inspired way of resolving the query.   
We present a modification to this method that gives the complete prime factorization of Carmichael numbers, which means the algorithm will be unconditionally correct.  
This will demonstrate that \textsc{carmichaels} is in \textbf{\textsc{p}} where only the run-time is conditioned on the ERH.
We conclude with a note on the bit complexity of \textsc{carmichaels}.
 
\subsection{Prior work}\label{subsec:query_is_hard}

In 1978, the authors of \cite{pom_self} said that it was ``usually much harder to show that a given large number is Carmichael than it is to show that it is a [base $b$ Fermat pseudoprime].''  
They noted the obvious approach of factoring $n$ to invoke Korselt's criterion and that $n$ may be hard to factor.  
Then they proposed an alternative way where divisors of $n-1$ are used to construct potential $p-1$ and thereby find factors of $n$.  
At first glance, the computational complexity of \textsc{carmichaels} is that of integer factorization (of either $n$ or $n-1$), which we account as a difficult problem. 

\subsection{The strong Fermat test} \label{subsec:primer_primes}

Miller used the following theorem to design a reliable primality test \cite{miller}.
\begin{theorem}[{\cite[Theorem 3.5.1]{cranpom}}] \label{thm:strong_fermat}  
Suppose $n$ is an odd prime and $n - 1 = 2^sn' $ where $n'$ is odd.  
If $b$ shares no common factors with $n$ then either $b^{n'} \equiv 1 \pmod{n}$, or there exists $i$ such  that $0 \leq i < s$ and $b^{2^in'} \equiv -1 \pmod{n}$.
\end{theorem}
\noindent Under the ERH, a witness for compositeness will be found for some $b < 2 \log^2 n$ \cite[Theorem 3.5.12]{cranpom}. This almost establishes \textsc{primes} is in \textbf{\textsc{p}}; correctness and runtime are contingent on the ERH.  

\subsection{\textsc{carmichaels} is in \textbf{\textsc{p}}, conditionally  } \label{subsec:car_p_erh}

We follow Miller's approach to establish that \textsc{carmichaels} is in \textbf{\textsc{p}}: use the Fermat test for every $b \in [2, 2\log^2 n)$.  
Failure to witness that $n$ is composite implies it is a Carmichael number\footnote{
Or $n$ could be an absolute pseudoCarmichael; that is, it could be prime.  
We assume that these algorithms are invoked on composite integers, and that composite inputs are guaranteed by some other means.}.  
The correctness and run-time are conditioned on the ERH.  

It is unfortunate that this does not provide the prime factors of $n$.
Providing prime factors of $n$ would be akin to a primality certificate;
verification could be done without re-doing all of the tests.
Since there are open questions about the prime factors of Carmichael numbers, 
an algorithm should provide them.  
If incorporated into a tabulation algorithm, historic precedence requires providing these factors. 
So, we provide an algorithm to do that.
This will give us an algorithm that is unconditionally correct, 
which is a greatly desired feature if used in a tabulation algorithm.   

\subsection{Factoring Carmichael numbers}\label{subsec:car_zpp}

It has long been known that it is simple to split $n$ if the Fermat test fails to detect $n$ is composite and the strong Fermat test detects $n$ is composite \cite{pom_goutier, monier}.  
Therefore, Carmichael numbers are easily split. 
To provide a complete prime factorization, splitting a number is often sufficient. 
We recursively invoke the algorithm on the smaller factors until we have the complete prime factorization.  
The Fermat splitting cannot proceed in this way; 
it is unlikely that a composite divisor will be a Fermat pseudoprime.
  
For the rest of the subsection, we assume that $n$ is a Carmichael number and we wish to provide a proof of that fact by exhibiting its prime factorization.  
We give a high-level description of how the splittings witness all possible factorizations, state the algorithm, and finally give a proof of the claims.  
In \cite{probpoly}, it is noted that the splitting procedure may not given a complete prime factorization and we show it can for Carmichael numbers.  

We can better understand how the splitting occurs via the (repeated) difference of squares factorization of $b^{n-1} - 1$. 
Let $n'$ be the largest odd divisor of $n-1$ and $s = \nu_2 (n-1)$.  Then 
\[ b^{n-1} - 1 = (b^{n'} - 1)\prod_{i = 0}^{s-1}(b^{2^i n'} + 1 ). \]
If $n$ divides the left side of the equation (Fermat test does not detect $n$ is composite), 
then its prime factors are found in the algebraic factors on the right side.  
When $n$ is a strong pseudoprime, all prime factors of $n$ align in a single algebraic factor.  
Otherwise, the prime factors are distributed in a predictable way across two or more factors.
Here is a visual representation of the $s = 3$ case.  

\[
  b^{8n'} - 1  = \underbrace{(b^{4n'} + 1)\overbrace{(b^{2n'} + 1)\underbrace{(b^{n'} + 1)(b^{n'} - 1)}_\text{$\nu_2(q-1) = 1$}}^{\text{$\nu_2(q-1) = 2$}}}_{\text{$\nu_2(q-1) = 3$}}
\]
If $n$ is a Carmichael number,  then $\nu_2(q-1) \leq 3$ for all $q|n$.   Fermat's little theorem says $b^{q-1} \equiv 1 \pmod{q}$ and the braces are showing where the divisibility occurs by the $2$-valuation of $q-1$.  
The divisibility will be in the leftmost factor whenever $b$ is not a square mod $q$.

We compute $\gcd( n, b^{2^i n'} + 1)$ for each $0 \leq i < s$ to split these prime factors from others.  
These gcd computations correspond to the exponentiations required for the strong Fermat test, 
%This set of $s$ gcd checks is equivalent to one strong Fermat test, 
and allows multiple non-trivial factors of $n$ to be found with a single $b$.
Previous literature on this problem (e.g., \cite{pom_goutier, monier}) focused only on the 
smallest valuation which was guaranteed to witness a splitting.

We have to keep track of the information obtained from the splittings.  
A high-level description of how to answer the query ``Is $n$ a Carmichael number?'' is that we iterate over different choices of $b$.  For each $b$ we will likely either  
\begin{enumerate}
    \item witness $n$ is composite with the weak test and answer ``no'', or
    \item progress towards a complete prime factorization with the strong test.
\end{enumerate}
We appeal to Korselt's criterion when we find a complete prime factorization of $n$.  More robustly, we have:  

\begin{algorithm}[H]
\caption{Returns the truth value of ``$n$ is a Carmichael number.'' }
\begin{algorithmic}[1]
\Require $n$, a positive composite integer
\State Initialize empty vectors \textsc{primes}, \textsc{next\_comp}, and \textsc{cur\_comp}.
\State Put $n$ into \textsc{cur\_comp}.
\While{\textsc{cur\_comp} is not empty}
    \State Choose $b$ as consecutive primes in $ [2, n-2]$ 
    \State Compute $X[i] \equiv b^{2^i n'} \pmod{n} $ for $0 \leq i < s$
    \State \textbf{if} $b^n \not \equiv b \pmod{n}$ \Return False 
    \ForEach{ $m$ in \textsc{cur\_comp} }
        \For { $i < s$}
            \State Calculate $g = \gcd(m, X[i] + 1)$
            \State $m = m/g$
            \State \textbf{if} $1 < g $ \textbf{then} put $g$ into \textsc{next\_comp} or \textsc{primes} 
        \EndFor
        \State \textbf{if} $1 < m $ \textbf{then} put $m$ into \textsc{next\_comp} or \textsc{primes} 
    \EndFor
    \State Swap \textsc{next\_comp} and \textsc{cur\_comp} 
\EndWhile
\State \Return  $\lambda(n) | (n-1)$  
\end{algorithmic}
\label{alg:car_in_p_cond}
\end{algorithm}

\textbf{Remarks:}
\begin{itemize}
\item Lines $5$-$6$ are all part of one square-and-multiply exponentiation ladder that represent the strong and weak Fermat tests.  
\item Lines $11$ and $13$ are, perhaps, understated.  
Putting these values in the vectors requires a primality check.  
Any time a composite check is done, we could also check that the number is not a nontrivial power, returning ``False'' if detected.  
Letting $P$ be the product of the known prime factors, we could also return ``False'' if $P$ is not cyclic or if $\lambda(P)\nmid (n-1)$.  
If this is done as part of lines $11$ and $13$, then line $17$ could read ``\textbf{return} True''.   
\item If prime or composite factors of $n$ are already known, initialize \textsc{primes} and \textsc{cur\_comp} with the respective factors.  
% \item It is also possible to do lines $5$ and $6$ with respect to $m$ (and use the CRT) so that the arithmetic is on smaller quantities.
\end{itemize}

%\begin{example}  Consider the base-$2$ Fermat pseudoprime $n = 1093^2$.  If prime-power detecting is implemented then Algorithm \ref{alg:car_in_p_cond} would return ``False'' as part of line $2$'s attempt to put $n$ into \textsc{cur\_comp}.  As literally written, Algorithm \ref{alg:car_in_p_cond} will return ``False'' with the Fermat test on line $6$ with $b = 3$ after failing to find a complete factorization with $b = 2$.
%\end{example}

\begin{example}
Consider $n =349407515342287435050603204719587201$,  the least Carmichael number with $20$ prime factors.
A number enclosed with parentheses denotes a composite and these numbers are in \textsc{cur\_comp} for the next choice of $b$.  
\begin{itemize}
    \item $b = 2$: \ $ n = 17 \cdot  97 \cdot 193 \cdot 641 \cdot (24261623)(8987) (7855279748239573). $ 
    \item $b = 3$: \ $24261623 = 71 \cdot 73 \cdot (4681)$, \  $8987 = 11 \cdot (817)$,  $7855279748239573 = 41 \cdot 113 \cdot (256477)(2257) (2929)$.
    \item $b = 5$: \ $4681 = (4681)$,  \  $817 = 19 \cdot 43$, \  $256477 = 13 \cdot (19729) $, \  $2257= 37 \cdot 61$, \ $2929 = 29 \cdot 101$.
    \item $b = 7$: \  $4681 = 31 \cdot 151$, \  $19729 = 109 \cdot 181$.
\end{itemize}
We have the complete prime factorization of $n$ and can verify that it is a Carmichael number with Korselt's criterion.  
\end{example}

\begin{repeatthm}{thm:iscarmichael}
\textsc{carmichaels} is in \textbf{\textsc{p}}, where correctness is unconditional and the run-time analysis is conditioned on the Extended Riemann Hypothesis (ERH). 
\end{repeatthm}

%With a concrete algorithm for reference, we are now ready to prove Theorem \ref{thm:iscarmichael}.

%\begin{theorem}
%\textsc{carmichaels} is in \textbf{\textsc{p}}, where only the run-time is conditioned on the ERH.  
%\end{theorem}
%\vspace{1em}
\begin{proof}
%\noindent \textit{Proof of Theorem \ref{thm:iscarmichael}.} 
Algorithm \ref{alg:car_in_p_cond} is unconditionally correct; lines $6$ and $17$ exit with no error. 
We need to show that it finishes in the required time constraint. 
Both cases use Theorem 1.4.5 of \cite{cranpom}.

Under the ERH and for composite $n$ that are not Carmichael numbers, a witness for compositeness is found (see line $6$) for some $b <  2\log^2 n$.  

We now establish that under the ERH and for $n$ that are Carmichael numbers, Algorithm \ref{alg:car_in_p_cond} finds the complete prime factorization from a set of $b < 2\log^2 n$.
It is sufficient to show that any two distinct primes may be split from each other (see line $9$);  
consider $m= q_1q_2$ a product of two distinct primes, a composite divisor of $n$ with $v_2(q_i - 1) = s_i \leq v_2(n-1)$ for $i = 1,2$.   

Case 1: When $s_1 = s_2$, the ERH guarantees the existence of a $b$ that is a square mod $q_1$ but not a square mod $q_2$ (or vice-versa) in the interval $[2, 2\log^2 q_1q_2)$.    

Case 2: If $s_1 < s_2$, it is sufficient to find $b$ that is not a square mod $q_2$.  The ERH says this will happen for some $b$ in the interval $[2, 2\log^2 q_2)$. 
%\qed
%\vspace{1em}
\end{proof}

What is the contrast between using Miller's test to address \textsc{primes} and using Algorithm \ref{alg:car_in_p_cond} to address \textsc{carmichaels}?  
Correctness.     
For Miller's test, a ``yes'' answer is contingent on the ERH.
For Algorithm \ref{alg:car_in_p_cond}, if the ERH is wrong it only means that we wait longer than we had expected.  

One can re-do all of Subsections \ref{subsec:primer_primes}-\ref{subsec:car_zpp}, and follow a Rabin-inspired approach to primality testing where $b$ are selected uniformly at random \cite{rabin}.  Everything follows through similarly except counting arguments (see \cite{pom_goutier, monier}) take the place of appealing to the ERH.  The end result would be the following theorem.

\begin{theorem}
\textsc{carmichaels} is in \textbf{\textsc{zpp}} (zero-error probabilistic polynomial time).  
\end{theorem}

The interested reader may find the details in {\cite[Theorem 3.9]{webster_CN}}.

\subsection{Bit complexity of \textsc{carmichaels}}
The average case corresponds to detecting $n$ is composite with the first base-$2$ Fermat test.  
So, we sketch a bound on the worst case.  
There is a factor of $\log^2 n$ coming from the number of Fermat bases.   
We can bound the loop in line $7$ by a factor of $\log n$ because that bounds the number of factors that can product to $n$.
We can bound the loop in line $8$ by a factor of $\log n$ because $s$ is bounded by the bit-length of $n$.
The largest cost per iteration of the while loop occurs in lines $11$ and $13$: prime testing required to place $g$ and $m$ in their respective vectors.  
The Lenstra-Pomerance variant \cite{LP} of AKS \cite{AKS}  accomplishes this in time $O(\log^6n \log_2^{c_o}n)$ for some effectively computable constant $c_o$.    
This gives a bit-complexity of  $O(\log^{10}n \log_2^{c_0}n)$.  

We doubt this result is sharp and we list three reasons.
First, it is likely that primality proofs can be delayed until the end;  it is easier to detect a number is composite than it is to provide a primality proof.  
In this case, the total number of primality tests would be proportional to the number of prime factors of $n$ rather than the number of times the while-loop executes.
Also, if we keep track of $\lambda(P)$ as it is constructed, it might be possible to incorporate the asymptotically cheaper BLS primality test \cite{bls} because we would know factored portions of $p-1$.
Second, the bound on line $7$ uses the worst case;  it is difficult for this worst case to be achieved and sustained.
Third, letting $s'$ be the bit-length of $n$, the bound on line $8$ is only rarely close to $s'$.   
The worst case can be improved using $p < \sqrt{n}$ implies $i < \min\{ s, s'/2\}$.  
More generally, as we learn prime factors of $n$, better estimates of $\max_{p|m}\{v_2(p-1)\}$ can be used, where $m$ corresponds to the remaining unfactored composite factors of $n$.  

\section{Algorithms akin to Kn\"odel's Argument - A survey }\label{sec:survey}

The first subsection provides a brief sketch of the two main algorithmic pieces used in prior tabulations.  
Subsection \ref{subsec:survey} surveys prior work that advanced the algorithmic theory of Carmichael tabulation.  
Subsection \ref{subsec:10_22} summarizes \cite{sw_ants3} to provide a unifying understanding of all prior tabulation algorithms.

\subsection{Algorithmic Basics}\label{subsec:basic_algs}

All of the below use Theorems \ref{thm:bounds_theorem} and \ref{thm:CD_theorem} in an algorithmic way.  
In \cite{sw_ants2, sw_ants3}, we gave two classes of algorithms that find $R$ so that $n = PR$ is a Carmichael number.  

In \cite{sw_ants2}, the focus was finding $R$ with exactly two prime factors and we showed they could be found in time almost linear in $P$.  
The key idea was to use Theorem \ref{thm:CD_theorem}.
Jaeschke and Pinch \cite{jaeschke_carm, pinch15} iterate over the quantities $D$ and $C$ to find $R= qr$ and this was effective for $P$ prime. 
In \cite{sw_ants2}, we showed that you can iterate over the quantities $D$ and $\Delta$ to handle composite $P$ nearly as efficiently. 
We also provided asymptotic analysis of these algorithms and argued that if $P < B^{1/3}$, it was better to invoke this algorithm than perform an exhaustive search for the two primes.

In \cite{sw_ants3}, the focus was finding $R$ as a prime.  
The key insight is that Part \ref{second_bullet} of Theorem \ref{thm:bounds_theorem} is a ``divisor in residue class'' problem and there are efficient algorithms that solve this problem \cite{div_in_class, div_in_class_con}.  
We showed that the set of cases which are difficult correspond to $P$ being small with respect to $B$, $\lambda(P)$ being small compared to $P$, or $P-1$ being highly divisible by $\lambda(P)$.  
We do not expect many of these difficult cases and the timing information at the end of \cite{sw_ants3} confirmed that the improvements had negligible impact.  
We refer to the class of algorithms that can find $R$ prime as \textbf{SPC} (single prime completion) algorithms.  
A well-written \textbf{SPC} algorithm should be no slower than checking the arithmetic progression $\mathcal{R}$ for primes.

\subsection{The Survey}\label{subsec:survey}

We do not think that any of the authors were inspired by Kn\"odel's counting argument for Carmichael numbers.  
For Carmichael numbers with $d$ prime factors, prime-by-prime completion methods will create $P_{d-1} = P_{d-2}p_{d-1}$.  
This cyclic number satisfies $P_{d-2}p_{d-1}^2 < B$ so that there is enough room for one more prime.  
The count of these numbers matches $\mathcal{K}(B)$ {\cite[Theorem 8]{sw_ants3}}
and serves as the asymptotic analysis of all of these algorithms since producing these inputs is the most expensive step.  
Our survey below focuses on algorithmic advancements. 
Several authors have extended the tabulation bound using known algorithms, including $B = 25 \cdot 10^9$ in \cite{pom_self}, $B = 10^{13}$ in \cite{keller}, $B =10^{14}$ in \cite{guthmann}, and $B = 10^{22}$ in \cite{goutier}.

\subsubsection{Swift's Tabulation} 
The first systemic tabulation of Carmichael numbers was done by J.D. Swift in 1975 on a CDC 1604 and his bound was $10^9$ \cite{swift}.  
Of note, he specifically precludes the approach found in the next section, writing ``The computer programs used depended explicitly on congruence properties of the CN's with respect to their component primes rather than on the pseudoprimality with respect to any particular base. 
A different routine was run for each possible number of primes occurring in the factorization.''   
For this tabulation, the maximum count of prime factors was six.  
This sets the stage for all tabulations that follow.   

%\subsubsection{Pomerance, Selfridge, and Wagstaff's Tabulation}
%It is worth mentioning that Pomerance, Selfridge, and Wagstaff found all Carmichael numbers less than $25 \cdot 10^9$ in 1980 \cite{pom_self}.  
%They were trying to find all kinds of different pseudoprimes and this tabulation was not specifically aimed at finding Carmichael numbers.  

\subsubsection{Jaeschke's Tabulation}
In 1990, Gerhard Jaeschke found all Carmichael numbers less than $10^{12}$ \cite{jaeschke_carm}.
Jaeschke had special routines for $d = 3$ and $d=4$ for when $P$ was small that were like the $CD$ method of Subsection \ref{subsec:basic_algs}.  
The set of inputs for his algorithm were preproducts $P$ such that $Pp < 10^{12}$. 

\subsubsection{Pinch's Tabulations}\label{subsec:pinch}

In 1993, Richard Pinch tabulated up to $10^{15}$ \cite{pinch15}.  
Over the next fourteens years he used the same algorithm to extend the bound by orders of magnitude up to $10^{21}$ \cite{pinch16, pinch17, pinch18, pinch21}.

One novel aspect of his tabulation was the ``large prime variation''  which created numbers divisible by a prime greater than $B^{1/3}$.  
This stage provided some inspiration for our own approach.  
He used the Fermat test as a rejection mechanism.  
For prime $P > B^{1/3}$, he created $n = PR$ with $R \in \mathcal{R}$ and for each one he tested if $2^n \equiv 2\pmod{n}$.  
For those that passed, he used trial division to factor $R$ in order to apply Korselt's criterion.  
He argued that using the Fermat test as a filter would cause the cost of factoring to be a lower-order cost in the context of a complete tabulation. 
We repeat his argument and offer a different perspective on his conclusion that is relevant for our present work.

The cost to perform the trial division\footnote{
The factor $R$ ranges in size from $B^{1/3}$ to $B^{2/3}$.} 
is $O(\sqrt{ B/P}) =  O(B^{1/3})$ for each $R$ coming from a given $P$. 
Let $C_B$, be the count that pass to the next stage.  
The count of all $PR$ is
\[ \sum_{B^{1/3} < P < B^{1/2}} \frac{B}{P(P-1)} = O(B^{2/3}).\]
The total cost is $O( B^{2/3} + C_BB^{1/3})$, where the first term is the cost to create and filter all inputs and the second term is the cost of factoring the $R$ for the base-$2$ Fermat pseudoprimes that pass the filter.  
He concludes that as $C_BB^{1/3}$ is ``noticeably smaller than $B^{2/3}$,  
the large-prime variation gives an improvement.''  

What are we counting with $C_B$?   
The base-$2$ Fermat pseudoprimes.  
The count of all base-$2$ Fermat pseudoprimes less than $B$ is conjectured to be of the form $B \exp \left\{  -\log B \log_3  B/ \log_2  B \right\}$.
We are concerned with a subset of size $O(B^{2/3})$ and we extrapolate the exponential term as a probability-like filter.
We conjecture that $ C_B \sim B^{2/3} \exp \left\{  -\log B \log_3 B / \log_2 B \right\}$.
Viewed this way, Pinch's claim that $C_BB^{1/3}$ is noticeably smaller than $B^{2/3}$ is not a product of asymptotic analysis but of the empirical rarity of pseudoprimes in the tabulation range.  
Theoretically,  $B^{2/3}$ is dominated by the nearly linear $C_BB^{1/3}$ term!  

An alternative expression for the cost is $O(B^{2/3} + C_BC_F)$, where $C_B$ is as before and $C_F$ is the cost to factor.
It should be clear than any of the general deterministic factoring algorithms will cause the $C_BC_F$ term to dominate.
However, if we use Algorithm \ref{alg:car_in_p_cond} then this term is, indeed, the lower-order term.

\subsection{Tabulation for \texorpdfstring{$B = 10^{22}$}{}  } \label{subsec:10_22}
We encourage the reader to consult \cite{pinch15, sw_ants2, sw_ants3} for a thorough explanation, 
and we give a concise description suitable for analysis and understanding.   
Our approach was bifurcated first by the size of $P$.
For large preproducts $P$ the tabulation was accomplished by a series of individual tabulations organized by the count of primes dividing $P$.  

Using the algorithms presented in \cite{sw_ants2}, the small preproduct regime considers all $P < X$ and constructs $n = Pqr$.  
The cost of this phase is $O((X \log X)^2)$.
So long as $X$ is chosen so that $(X \log X)^2$ is dominated by $\mathcal{K}(B)$, this regime is a lower-order cost in the asymptotic analysis.  
In practice, we chose $X$ to slightly exceed $B^{1/3}$ and, by doing so, this regime finds all Carmichael numbers less than $B$ with exactly $3$ prime factors.    
The choice of $X$ does not impact the correctness of what follows.  

The large preproduct regime constructs $P$ in a prime-by-prime fashion.  
Here, the tabulation regime is further distinguished by the count of primes in a preproduct.  
We considered $\omega(n) = d$ for $4 \leq d \leq 13$:  ten different tabulations.
We give pseudocode for the $d=5$ algorithm and then give a unifying description of all the tabulation regimes.

\begin{algorithm}[H]  %[H] to force placement, maybe not needed later 
\caption{Carmichael numbers $n < B$ with $\omega(n) = 5$}
\begin{algorithmic}[1]
\Require The tabulation bound $B$ and the cross over bound $X$
\ForEach{ $q_1$ in $(1, \sqrt[5]{B})$ }
    \State Let $P_1 = q_1$
    \ForEach{ $q_2$ admissible to $P_1$ in $(q_1, \sqrt[4]{B/P_1})$ }
        \State Let $P_2 = P_1q_2$
        \ForEach{ $q_3$ admissible to $P_2$ in $(\max \{q_2, X/P_2\}, \sqrt[3]{B/P_2})$ }
            \State Let $P_3 = P_2q_3$
            \ForEach{ $q_4$ admissible to $P_3$ in $(q_3, \sqrt{B/P_3})$ }
                \State Call \textbf{SPC} on $P_4 = P_3p_4$
            \EndFor
        \EndFor
    \EndFor
\EndFor
\end{algorithmic}
\label{alg:prime_by_prime}
\end{algorithm}
\noindent Note $P_3$, which has enough room for two more primes, exceeds $X$.  
In this way, we do not duplicate work associated to the $P < X$ regime.  

In general, the outer loops correspond to creating preproducts missing more than one prime:  
for $k > 1$ and at the $d-k$ level, a search commences for primes $q$ in $(p, \sqrt[k]{B/P_{d-k}})$ admissible to $P_{d-k}$. 
At the $d-3$ level, we need to ensure that the preproduct constructed exceeds $X$.  
Then every distinct tabulation regime always had three specific levels:
\begin{enumerate}
    \item The $d - 3$ level: $P_{ d - 3 }$ searches for $q$ up to $\sqrt[3]{B/P_{d-3}}$.  
    The starting point for $q$ is the larger of $X/P_{d-3}$ or $p$ to ensure $P_{d-2} = P_{d-3}q > X$.  
    \item The $d - 2$ level: $P_{ d - 2 } > X$ searches for $q$ in $(p, \sqrt{B/P_{d-2}} ).$  
    \item The $d-1$ level is the inner-loop and calls \textbf{SPC}.  
\end{enumerate}
A given preproduct $P$ may appear in multiple tabulation regimes.  
If $P$ has enough room for three more primes, then it also has enough room for two more primes and one more prime.  
However, it is important to know when a given preproduct will have only enough room for one more prime.  
If $P_{d-1} = P_{d-2}q$ where $q \geq \sqrt[3]{B/P_{d-2}}$, then $P_{d-1}$ can only be completed with a single prime. 

\subsection{Observed Problems with Prior Tabulations }\label{subsec:problems}

The asymptotically faster \textbf{SPC} method for the inner loop proved to be ineffective in practice.  
The common case was $P_{d-1}\lambda(P_{d-1}) > B$.  
When this happened, \textbf{SPC} frequently consisted of two steps: 
compute $r^\star$ and discard $Pr^\star$ for being too large.  
There were other downstream effects. 
First,  the cost of getting to the inner loop exceeded the work done in the inner loop.   
This meant that future implementations would be focused on the creation of cyclic products.
This is another hint that our approach was not optimal; the algorithmic realization of the counting argument of these products is not done via a prime-by-prime approach.
Second, this made parallel distribution of the work difficult.  
We chose to distribute at the $\lfloor d/2 \rfloor$ level which was an ad hoc choice made in an attempt to balance two key problems.  
Distributing it higher up was unbalanced and distributing it further down duplicated too much work.  

We summarize the two tabulations regimes in the following way.  
\begin{enumerate}
    \item The small case: ``Tabulate all $n = PR$ with $P < X$, and $\omega(R) = 2$.''
    \item The large case: ``Tabulate all $n = PR$ with $P > X$, and $\omega(R) = 2$.''
\end{enumerate}
The former step had run-time roughly quadratic in $X$ and the latter had run-time matching  Kn\"odel's bound.  
The small case removed an asymptotically small set of inputs from the large case.  
However, completing these inputs as if they were in the large case would have been expensive.
We conclude with a practical example to demonstrate the power of Algorithm \ref{alg:car_in_p_cond}.

\begin{example}\label{mot_example}
Let $P = 101 \cdot 103 \cdot 107 \cdot 109 \cdot 113 \cdot 127$ where $\lambda(P) = 68115600$ and $B = 10^{24}$.   How can we find all $n = PR < B$?  We now have two options:

\begin{enumerate}
    \item Prime-by-prime appending in $7 \leq d \leq 11$ tabulations:  
    \begin{enumerate}
        \item For $d = 7$: $P$ is the only preproduct at the $d-1$ level.
        \item For $d = 8$: $q$ searched for in the interval $(127, 757834)$ and there are $57255$ products $P_{d-1} = Pq$.
        \item For $d = 9$:  $1145658$ products $P_{d-1} = Pqr$.
        \item For $d = 10$: $1227386$ products $P_{d-1} = Pqrs$.
        \item For $d = 11$: $24849$ products $P_{d-1} = Pqrst$.
    \end{enumerate}
    The \textbf{SPC} algorithm is invoked $2455149$ times and there is significant overhead costs in creating those inputs.
    \item There are only $8431$ candidates $n = P(r^\star + k\lambda (P)) < B$.  Simply invoke Algorithm \ref{alg:car_in_p_cond} on each candidate.  
\end{enumerate}
\end{example}

Simply invoking Algorithm \ref{alg:car_in_p_cond} raises new problems.  
Part of the prime-by-prime approach is forced on us for reasons of correctness.  
In order to invoke Algorithm \ref{alg:car_in_p_cond} at scale, we will need a new way to account for the correctness of our algorithm.  
We address these issues in the next section.  

\section{Asymptotically faster algorithms} \label{sec:faster1}

We believe that the $10^{22}$ tabulation did too much work.  
From a pragmatic point of view, we should have had an early abort that would have prevented more primes being appended to $P$ when $P\lambda(P) > B$.  
From a theoretical point of view, the count of cyclic numbers that serve as inputs to the algorithm should be more like $\mathcal{E}(B)$.
We change how we view the two tabulation regimes of Subsection \ref{subsec:problems}.
The new way to consider the two cases is: 
\begin{enumerate}
    \item The small case: ``Tabulate all $n = PR$ with $P < X$ and $\omega(R) = 2$.''
    \item The large case: ``Tabulate all $n = PR$ with $P < X$ and $\omega(R) \geq 3$.''
\end{enumerate}
This conceptual framing of the problem would not have made sense to us prior to discovering Algorithm \ref{alg:car_in_p_cond}; 
the way to achieve $\omega(R)\geq 3$ was with prime-by-prime appending as sketched above.  
While Algorithm \ref{alg:car_in_p_cond} can simply be invoked on every candidate, we state an improved version based on sieving to reduce the amount of work required.

\subsection{ \texorpdfstring{$\lambda(P)$}{}-sieving } \label{subsec:l_sieve}

Given a cyclic number $P$, how do we find $R$ so that $n= PR$ is a Carmichael number?
Previously, tabulation algorithms constructed $P$ in factored form and the search was for an $R$ that was prime.  
The naive thing to do would be to check $\mathcal{R}$ for primes. 
Now, it is almost as efficient (the $\log n$ factors may differ) to search for composite $R$ by using Algorithm \ref{alg:car_in_p_cond}.
Having the prime factorization of $P$ just means that there is less work for Algorithm \ref{alg:car_in_p_cond} to do.
Further, we typically require that the primes dividing $R$ all exceed $p$.
Since these are in an arithmetic progression, sieving removes candidates faster than modular exponentiation.  

\subsubsection{ Controlling the prime divisors of  \texorpdfstring{$R$}{} }

    We may sieve by
\begin{enumerate}
    \item Primes:
    \begin{enumerate}
        \item The typical scenario is that we are appending to $P$ primes exceeding $p$, which means that we can sieve by any prime $q \leq p$ (including the prime divisors of $P$).
        \item We can sieve by any prime $q$ that is inadmissible to $P$.
    \end{enumerate}
    \item Composites:
    \begin{enumerate}
        \item We can sieve by the square of any prime.
        \item More generally, we can sieve by any inadmissible product.  For example, if $q$ and $2q + 1$ are both primes ($2q+1$ is not admissible to $q$), then we may sieve by $q(2q + 1)$. E.g., even if $3$ and $7$ are both admissible to $P$, $21$ cannot divide any Carmichael multiple of $P$.
    \end{enumerate}
\end{enumerate}
We refer to finding Carmichael numbers in this fashion as a \textbf{$\lambda(P)$-sieve}.  
\vspace{1em}

\noindent \textbf{Example \ref{mot_example} (cont'd).}  
\textit{Let $P = 101 \cdot 103 \cdot 107 \cdot 109 \cdot 113 \cdot 127$ where $\lambda(P) = 68115600$ and $B = 10^{24}$.   
There are only $8431$ candidates $n = P(r^\star + k\lambda (P)) < B$.  
There are $3886$ $k$ values that have $r^\star + k\lambda (P)$ divisible by a prime less than $129$ which are found with $\lambda(P)$-sieving. 
So, Algorithm \ref{alg:car_in_p_cond} is only invoked on $4545$ candidates.}

\subsection{Partitioning Carmichael numbers by preproducts }\label{subsec:partition}

Let $(P, \lambda(P), b)$ be the $3$-tuple representing the set of Carmichael numbers of the form $n = PR$ where the prime divisors of $R$ exceed $b$ or $R = 1$.  Observe a partition of Carmichael numbers:  

\begin{enumerate}

    \item $( 3, 2, 5)$: $n$ is a multiple of $3$ and the other primes exceed $5$.
    \item $( 5, 4, 5)$: $n$ is a multiple of $5$ and the other primes exceed $5$.
    \item $( 15, 4, 5)$: $n$ is a multiple of $15$ and the other primes exceed $5$.
    \item $( 1, 1, 5)$: $n$ has least prime factor exceeding $5$. 
\end{enumerate}

Our set of jobs is of the form $(P, \lambda(P), \max \{ X/P, p \})$ where $P < X$ and $Pp^3 < B$.  
These are the preproducts that can allow three or more primes to be appended.  
This is a subset of the preproducts that would have been found at the $P_{d-3}$ level.  
The key idea is that once we append a prime to a given $P$ we do not want to create a preproduct already in the initializing set.  
Either $p$ is large enough so that $Pp > X$ or we set the append bound to $X/P$ so that the first prime appended to $P$ causes the preproduct to exceed $X$.    

\subsection{An improved tabulation algorithm}

With the set of initializing preproducts, we call the following algorithm on each one.  

\begin{algorithm}[H]  %[H] to force placement, maybe not needed later 
\caption{Carmichael numbers $PR$ with $\omega(R) \geq 3$}
\begin{algorithmic}[1]
\Require The tabulation bound $B$, crossover bound $X$, $P$, and $\lambda(P)$
\If{  $P \lambda(P)^2 > B$ } 
        \State \textbf{$\lambda(P)$-sieve}   
\Else
    \ForEach{ prime $q$ admissible to $P$ in $(\max \{p, X/P\}, \sqrt[3]{B/P})$ }
        \State Recursively call this algorithm with preproduct $Pq$
    \EndFor
    \If{ $P > X $}
        \ForEach{ prime $q$ admissible to $P$ in $ (\sqrt[3]{B/P},  \sqrt{B/P})$ }
            \State  Call \textbf{SPC} on  $Pq$ 
        \EndFor
        \If{ $P > pX$ }
            \State Call \textbf{SPC} on  $P$      
        \EndIf
    \EndIf
\EndIf
\end{algorithmic}
\label{alg:tabulation}
\end{algorithm}

\begin{theorem}
Algorithm \ref{alg:tabulation} is correct.
\end{theorem}

\begin{proof}
The branching rule found in line $1$ is irrelevant for the correctness of this algorithm but is important for the run-time.  Each branch is capable of independently producing a correct tabulation.

The first branch (line $2$) will produce a correct tabulation because it is an exhaustive search.

The prime-by-prime branch (lines $4$-$14$) treats a given preproduct having more than two primes to append (lines $4$-$6$), having exactly two primes to append (lines $8$-$10$), and as having exactly one prime to append (line $12$).  So this finds all possible completions in a prime-by-prime manner.  
\end{proof}

\textbf{Remark:}
The $\lambda(P)$-sieving branch can find $R$ with $\omega(R) < 3$. 
If it is necessary to avoid this we could 
discard duplicates in the post-processing sorting of the entire tabulation or
we could track the recursion depth.  

Since line $1$ of Algorithm \ref{alg:tabulation} played no role in the correctness of the algorithm, its utility is as a heuristic to complete a given preproduct as efficiently as possible.  
The heuristic comes from balancing $B/(P\lambda(P))$ (a crude linear estimate of $\lambda(P)$-sieving) against $\sqrt{B/P}$ (a crude estimate of the cost of the prime-by-prime branch).  

We forgo a traditional analysis of this algorithm in order to state a more optimal algorithm.  
However, we make one comparison.  
For our initializing set of preproducts, it is impossible for the first branch to be taken.  
Then every initializing $P$ gets expanded into $Pq$ and recursively called (lines 4-6). 
This set of preproducts satisfies  $B^{1/3} < Pq < B^{2/3}.$  
The lower bound comes from the fact that the expansion set creates products greater than $X$ which we have chosen to be $B^{1/3}$.
The upper bound comes from the fact if $P$ is a prime just less than $B^{1/3}$, then the expansion is to a product nearly $B^{2/3}$.
The products that land close to $B^{1/3}$ in the initial expansion step are likely expanded again because they will likely not achieve the $P\lambda(P)^2 > B$ requirement.  
Our lower bound is likely weaker than it could be and the lower bound is what is important.
By assuming that all Carmichael numbers are found by $\lambda(P)$-sieving, we could bound the cost with the sum
\[ \sum_{B^{1/3} < P < B^{2/3} }  \frac{B}{P\lambda(P) }.  \]
Contrast this with 
\[ \sum_{B^{1/2} < P < B^{2/3} }  \frac{B}{P\lambda(P) }  \]
which is the key sum that Erd\H{o}s considered in \cite{erdos_CN} to bound the count of Carmichael numbers by
\[ B \exp \left\{ \frac{ -k \log B \log_3 B}{\log_2 B} \right\} \]
for some positive constant $k$. 
We suspect that by using the lower bound of $B^{1/3}$ that the same argument will go work but with a worse implied $k$ value.  
We forgo trying to prove this because we state a better algorithm that has an explicitly provided $k$ value.

\subsection{A Heuristically Optimal Algorithm }\label{sec:faster2}

Counting Carmichael numbers in the arithmetic progression $P(r^\star + k \lambda(P))$ is central to the argument in the proof of Theorem \ref{thm:CN_count}.  
Since $\lambda(P)$-sieving (in conjunction with Algorithm \ref{alg:car_in_p_cond}) allows efficient identification of Carmichael numbers in that arithmetic progression, the counting argument has an algorithmic realization.

%\begin{theorem}
%For each $\epsilon > 0$, there is an $B_0(\epsilon)$ such that for all $B > B_0(\epsilon)$, there is an algorithm that can find all Carmichael numbers less than  $B$ in time $B\exp(-(1-\epsilon)\log x \cdot \log_3 x / \log_2 x )$. 
%\end{theorem}

\begin{repeatthm}{thm:optimal}
For each $\epsilon > 0$, there is an $B_0(\epsilon)$ such that for all $B > B_0(\epsilon)$, there is an algorithm that can find all Carmichael numbers less than $B$ in at most $\mathcal{E}(B)$ 
queries of ``Is $n$ a Carmichael number?''.  
\end{repeatthm}

\begin{proof}
%\textit{Proof of Theorem \ref{thm:optimal}.} 
There are three sets of Carmichael numbers $n \leq B$ that are counted.  We show that there exists an algorithm that tabulates each of these sets in time proportional to the size of the set times the cost of answering the query, ``Is $n$ a Carmichael number?''
Since this cost is polynomial time, it remains to show that the inputs can be generated as claimed.  The three sets are
\begin{enumerate}
    \item $n \leq B^{1 - \delta}$,
    \item $B^{1 - \delta} < n \leq B $ and $n$ has a prime factor $p \geq B^{\delta}$, and
    \item $B^{1 - \delta} < n \leq B $ and all the prime factors of $n$ are below $B^{\delta}$.
\end{enumerate}
In the proof of Theorem \ref{thm:CN_count}, $\delta$ is eventually set to be $\epsilon/5$ to complete the proof but we describe the sets in terms of the parameter $\delta$.    
\begin{enumerate}
    \item  The first set represents the small Carmichael numbers.  
    This is a tabulation problem with a smaller upper bound.  
    We view this algorithm as recursive and apply the method of the other two sets to find these numbers if necessary\footnote{
    In reality, this set may be provided by some prior tabulation.}.  
    
    \item The count of these Carmichael numbers is bounded by considering sets $P( r^\star + k \lambda(P)) < B$ indexed by $k$, where $P = p$ is a prime exceeding $B^{\delta}$.  
    
    Algorithmically, use a sieve to find primes in $(B^{\delta}, \sqrt{B}]$.  For each prime $P > B^{\delta}$, use $\lambda(P)$-sieving.  The total cost is bounded by  
\[ \sum_{P > B^{\delta}} \frac{B}{P(P-1)}  < 2B^{1-\delta}\]
invocations of Algorithm \ref{alg:car_in_p_cond} 
    which is a lower order cost compared to set 3.
    
    \item  The count of these Carmichael numbers is bounded  by considering $P( r^\star + k \lambda(P))< B $, where $B^{1 - 2\delta} < P \leq B^{1 - \delta} $ are cyclic numbers that are $B^{\delta}$-smooth.  
    
    Algorithmically, we could use a sieve on this interval to find the $B^\delta$-smooth numbers and this would suffice for the purposes of our proof.  
    However, we recommend constructing the set explicitly. 
    This is the prime-by-prime construction technique but restricted to a table of primes less than $B^\delta$.
    In this way, we could have early-abort by not proceeding whenever $P\lambda(P) > B$.  
    
    Similar to the previous part, for each $P$ we use $\lambda(P)$-sieving.   The number of candidates created here is $B\exp(-(1-\epsilon)\log x \cdot \log_3 x / \log_2 x )$.  And this is what contributes the asymptotic cost given by the work in \cite{pom_self}.
\end{enumerate}

This demonstrates that the upper bound can be algorithmically realized. 
%\qed
%\vspace{1em}
\end{proof}

If the bound in Theorem \ref{thm:CN_count} is also a lower bound, as a heuristic argument shows it is, then the algorithm in the proof of Theorem \ref{thm:optimal} is optimal up to the cost of answering the query ``is $n$ a Carmichael number?''

We note two interesting features about this theoretical algorithm.  
First, only $\lambda(P)$-sieving is used. 
Second, since only $\lambda(P)$-sieving is used, it is easily adapted to tabulating Carmichael numbers in an interval: choose the $k$ so that $P(r^\star + k\lambda(P))$ only lands in the given interval.  
It has been traditional to re-do the entire tabulation as an independent check of prior work.  
Suppose we want to forgo this tradition.  
If we trust the correctness of the current tabulation, then we may provide a tabulation of all Carmichael numbers up to $10^{25}$ by tabulating Carmichael numbers in the interval $[10^{24}, 10^{25}]$.

\section{Carmichael numbers up to \texorpdfstring{$10^{24}$}{} }\label{sec:the_tab}

We estimate the large case for the $10^{22}$ computation from \cite{sw_ants3} took about $104$ days on $32$ Xeon E5-2630 processors ($192$ total cores).  
This was the prime-by-prime appending routine that had different tabulation regimes for each distinct prime count.  
In contrast, we estimate that Algorithm \ref{alg:tabulation} completed the $10^{24}$ case in about $97$ days on $5$ AMD EPYC 9734 processors ($560$ total cores). 

In terms of core hours, we achieved a bound that was $100$ times greater with about $2.72$ times the number of core hours.  
This is not a fair comparison since the two processors are several generations apart\footnote{
PassMark (\url{https://tinyurl.com/m3nm8vdy}) shows each EPYC 9734 thread performs about $1.8$ times the work of one E5-2630 thread.}.  
We suspect the recent computation did about $5$ times more work to achieve a bound $100$ times greater.
The $5$ could possibly be even lower because the parameter $X$ was more favorable to the $10^{22}$ tabulation. 
That is,  for $B = 10^{22}$ we had $X = 7 \cdot 10^7 \approx B^{0.3566}$, and for $B = 10^{24}$ we had  $X = 1.25\cdot 10^8  \approx B^{0.3373}$.  

Serial tabulations were performed by the codebase from \cite{sw_old_code} and the new codebase \cite{sw_new_code} on a single core of an AMD Ryzen 9 3900X processor.  
It is encouraging to see that the asymptotic crossover is realized already by $B = 10^{18}$.
Even more encouraging to note is that Algorithm \ref{alg:tabulation} has less duplicated work in its parallelization scheme than the ad hoc parallel distribution for the prime-by-prime approach.
%The timings in Table \ref{table:tab_timings} below are measured in seconds.

\begin{table}
\centering
\begin{tabular}{|c|c|c|}
  \hline
$B$ &  Prime by prime & Algorithm \ref{alg:tabulation} \\ \hline \hline
$10^{14}$ & 64 & 140 \\ \hline
$10^{15}$ & 397 & 711 \\ \hline
$10^{16}$ & 2456 & 3631 \\ \hline
$10^{17}$ & 15361 & 18352 \\ \hline
$10^{18}$ & 95476 & 92891 \\ \hline
\end{tabular}
\caption{Timing comparison (in seconds)}
\label{table:tab_timings}
\end{table}

\subsection{Statistics on the tabulation }\label{subsec:stats}

In the tables, we provide various statistics and counts on the tabulation.
Our tables omit the lower order of magnitudes, which may be found in the previous literature or in {\cite[Section 7.2]{webster_CN}}.

We start with the counts of Carmichael numbers by order of magnitude.  
We include four supplementary functions related to the lower-bound heuristic counting function.  
In \cite{dist_psp}, a more precise version of $\mathcal{E}(B)$ was stated:  
\[  B \exp{\left\{ -\frac{\log{B}}{\log_2{B}}\left( \log_3B + \log_4B + \frac{\log_4B - 1}{\log_3 B} + O\left( \left( \frac{ \log_4B}{\log_3{B}}\right)^2\right) \right) \right\} }\enspace,\]
which served as an upper bound and a heuristic lower bound.  
We provide four numeric quantities associated with the above heuristic.

The first function is the approximation $C(B)  = B^{j_1(B)}$.  
The expectation is that $j_1(B) \rightarrow 1$ as $B \rightarrow \infty$.  
The second function is $j_2(B)$, where
\[ C(B) = B^{1 - j_2(B)(\log_3B + \log_4B)/\log_2B} \enspace.\]
The expectation is that $j_2(B) \rightarrow 1$ as $B \rightarrow \infty$.  
The third is $j_3(B)$, where 
\[ C(B) = B \exp{\left\{ -\frac{j_3(B)\log{B} \log_3 B}{\log_2{B}} \right\} }\enspace.\]
The expectation is that $j_3(B) \rightarrow 1$ as $B \rightarrow \infty$.  
Finally, consider the function $j_4(B)$ where, 
\[ C(B) = B \exp{\left\{ -\frac{\log{B}}{\log_2{B}}\left( \log_3B + \log_4B + j_4(B)\right)\right\} }\enspace.\]
The expectation is that $j_4(B) \rightarrow 0$ as $B \rightarrow \infty$.

\begin{table}
\centering
\begin{tabular}{|c|r|c|c|c|c|c|}
  \hline
  $B$ & $C(B) $ & $j_1(B)$ & $j_2(B)$ & $j_3(B)$ & $j_4(B) $   \\ \hline \hline 
  $10^{16}$ &    $246683$ & 0.337008 & 1.561018 & 1.864056 & 0.859364  \\ \hline
  $10^{17}$ &    $585355$ & 0.339259 & 1.551898 & 1.864722 & 0.861721  \\ \hline
  $10^{18}$ &   $1401644$ & 0.341479 & 1.543801 & 1.865223 & 0.863922   \\ \hline
  $10^{19}$ &   $3381806$ & 0.343639 & 1.536592 & 1.865646 & 0.866053  \\ \hline
  $10^{20}$ &   $8220777$ & 0.345745 & 1.530103 & 1.865977 & 0.868077  \\ \hline
  $10^{21}$ &  $20138200$ & 0.347810 & 1.524189 & 1.866191 & 0.869947  \\ \hline
  $10^{22}$ &  $49679870$ & 0.349826 & 1.518777 & 1.866320 & 0.871694  \\ \hline
  $10^{23}$ & $123381982$ & 0.351793 & 1.513795 & 1.866375 & 0.873323  \\ \hline
  $10^{24}$ & $308279939$ & 0.353706 & 1.509200 & 1.866382 & 0.874867  \\ \hline
\end{tabular}
\caption{Count of Carmichael numbers and numerical estimates}
\label{table:numerical_estimates}
\end{table}

The functions $j_3$ and $j_4$ appear as $j$ and $k$ in the table in \cite{dist_psp}.  
That table went up to $25 \cdot 10^9$.  
In that table and for the next few orders of magnitude, their values are trending in the expected direction. 
Thereafter, these two quantities trend away from their expected values.  
We are unsure how we could use $j_3$ or $j_4$ in a predictive way if they are trending in the wrong direction.
If anything might be conjectured from $j_3$ and $j_4$, it would be that the lower order terms in the exponential expression are incorrect as lower bounds; 
we believe the table is too small for them to behave consistently with the asymptotic expectations.
To make predictions we will have to use $j_1$ or $j_2$. 

In \cite{two_contra}, the authors speculate that $C(B)$ will not exceed $B^{1/2}$ for any $B < 10^{100}$. 
Based on the data above, we conjecture that the threshold will not be crossed while $B < 10^{150}$.  
Justification for this conjecture will depend specifically on $j_2$.  
First note that the rate of decrease for $j_2$ is itself decreasing; 
$j_2$ decreased by $0.00591, 0.00541, 0.00498$, and $0.00459$ across the last $5$ orders of magnitude.  
If $C(10^{100})= 10^{50}$ then $j_2(10^{100}) \approx 1.2248$ and if $C(10^{150}) = 10^{75}$ then $j_2(10^{150}) \approx 1.2521$.
%Note that $j_2(10^{100}) \approx 1.2248$ corresponds to $C(10^{100})= 10^{50}$ and $j_2(10^{150}) \approx 1.2521$ corresponds to $C(10^{150}) = 10^{75}$.  
Meeting $1.2248$ across $76$ orders of magnitude yields an average decrease of $0.0037$, while meeting $1.2521$ across $126$ orders of magnitude yields an average decrease of $0.002$.  We believe that the latter is more likely, as it is further away from the smallest observed difference of $0.00459$.

Tables \ref{table:count_small} and  \ref{table:count_large} contain the distribution of Carmichael numbers by the count of prime factors. 
The columns represent the prime factor counts and the rows are by orders of magnitude.

\begin{table}
\centering
\begin{tabular}{|c|c|c|c|c|c|c|}
  \hline
 $B$ &   $d = 3$ &  $4$ &  $5$ &  $6$ &  $7$ &  $8$  \\  \hline \hline
$10^{16}$ &  $10816$ &  $16202$ &  $55012$ &  $86696$ &  $60150$ &  $16348$  \\  \hline 
$10^{17}$ &  $19539$ &  $25758$ &  $100707$ &  $194306$ &  $172234$ &  $63635$  \\  \hline 
$10^{18}$ &  $35586$ &  $40685$ &  $178063$ &  $414660$ &  $460553$ &  $223997$  \\  \hline 
$10^{19}$ &  $65309$ &  $63343$ &  $306310$ &  $849564$ &  $1159167$ &  $720406$  \\  \hline 
$10^{20}$ &  $120625$ &  $98253$ &  $514381$ &  $1681744$ &  $2774702$ &  $2148017$ \\  \hline 
$10^{21}$ &  $224763$ &  $151566$ &  $846627$ &  $3230120$ &  $6363475$ &  $6015901$  \\  \hline 
$10^{22}$ &  $420658$ &  $232742$ &  $1370257$ &  $6034046$ &  $14056367$ &  $16005646$  \\  \hline 
$10^{23}$ &  $790885$ &  $357078$ &  $2181324$ &  $11007792$ &  $30038569$ &  $40696220$ \\  \hline 
$10^{24}$ &  $1494738$ &  $548265$ &  $3429390$ &  $19665194$ &  $62338997$ &  $99469866$  \\  \hline 
\end{tabular}
\caption{Counts by prime factors $d < 9$}
\label{table:count_small}
\end{table}

In \cite{two_contra}, it is conjectured that $C_k(B) > C_{k+1}(B)$ should hold once $B$ is large enough. 
At this range, the inequality only holds for $k = 3$ and $ k > 8$.  
We conjecture that we are in the asymptotic domain for $k=3$;
that is, we think this inequality will not reverse again.
We also think that we are close to $C_3(B) > C_5(B)$ and that this inequality will be obtained before $C_4(B) > C_5(B)$ is obtained.  
We do not think we are in the asymptotic range for the $k > 8$ inequalities and that these inequalities exist only because $B$ is so small that we have not found enough Carmichael numbers with $k > 8$ prime factors.  
We conjecture that each of these inequalities (for $k > 8$) will reverse (before reversing again).  
There is not enough data in the table for us to feel comfortable conjecturing when $C_4(B) > C_5(B)$.

\begin{table}
\centering
\begin{tabular}{|c|c|c|c|c|c|c|}
  \hline
$B$ & $d = 9$ &  $10$ &  $11$ &  $12$ &  $13$ &  $14$ \\  \hline \hline
$10^{16}$ &$1436$ &  $23$ &  $0$ &  $0$ &  $0$ &  $0$ \\  \hline 
$10^{17}$ &$8835$ &  $340$ &  $1$ &  $0$ &  $0$ &  $0$ \\  \hline 
$10^{18}$ &$44993$ &  $3058$ &  $49$ &  $0$ &  $0$ &  $0$ \\  \hline 
$10^{19}$ &$196391$ &  $20738$ &  $576$ &  $2$ &  $0$ &  $0$ \\  \hline 
$10^{20}$ &$762963$ &  $114232$ &  $5804$ &  $56$ &  $0$ &  $0$ \\  \hline 
$10^{21}$ &$2714473$ &  $547528$ &  $42764$ &  $983$ &  $0$ &  $0$ \\  \hline 
$10^{22}$ &$8939435$ &  $2347828$ &  $262818$ &  $10018$ &  $55$ &  $0$ \\  \hline 
$10^{23}$ &$27660029$ &  $9178659$ &  $1388714$ &  $81434$ &  $1277$ &  $1$ \\  \hline 
$10^{24}$ &$81041495$ &  $33199094$ &  $6529921$ &  $547590$ &  $15328$ &  $61$ \\  \hline 
\end{tabular}
\caption{Counts by prime factors $d \geq 9$}
\label{table:count_large}
\end{table}

One question explored in \cite{two_contra} regards imprimitive Carmichael numbers.  
The motivation was Chernick's family where if $6m + 1, 12m + 1,$ and $18m + 1$ are all prime then their product is a Carmichael number.  
This example may be generalized. 
Every Carmichael number may be used to generate an infinite family of Carmichael numbers by considering the successive ratios of $p_i -1$ for all primes dividing $n$ under the assumption that these ratios may be achieved by prime-tuples infinitely often.
The least Carmichael number in a given family is primitive. 
For example, $7 \cdot 13 \cdot 19$ and $37 \cdot 73 \cdot 109$ both have the same set of successive ratios (arising from Chernick's family) but only the first is primitive.  

In {\cite[Corollary 4]{two_contra}}, the authors prove that the set of imprimitive Carmichael numbers is $o(B^{1/3})$. 
In both \cite{harman2, lichtman}, the authors establish that the count of all Carmichael numbers exceeds $B^{1/3}$.  
Combining these two results shows that the set of imprimitive Carmichael numbers forms a relative $0$-density subset of all Carmichael numbers\footnote{
We thank Carl Pomerance for pointing this out.}. 
We extend the table of imprimitive Carmichael numbers found in \cite{two_contra}. 
The column labeled ``\%'' contains the percentage of Carmichael numbers that are imprimitive. 
These percentages are decreasing and these numbers make up $\approx 0.43 \%$ our total tabulation, which is concordant with a $0$-density subset.  
There are a total of $1312179$ imprimitve Carmichael numbers less than $10^{24}$.  

\begin{table}
\centering
\begin{tabular}{|c|c|c|c|c|c|c|}
\hline
$B$ &  \% & $d = 3$ &  $4$ &  $5$ &  $6$ &  $7$  \\  \hline \hline
 $10^{16}$ & $ 2.921$ & $6615$ &  $563$ &  $30$ &  $0$ &  $0$ \\  \hline 
 $10^{17}$ & $ 2.363$ & $12725$ &  $1036$ &  $72$ &  $1$ &  $0$ \\  \hline 
 $10^{18}$ & $ 1.880$ & $24396$ &  $1797$ &  $156$ &  $4$ &  $0$  \\  \hline 
 $10^{19}$ & $ 1.489$ & $46877$ &  $3193$ &  $264$ &  $8$ &  $1$  \\  \hline 
 $10^{20}$ & $ 1.169$ & $89854$ &  $5745$ &  $501$ &  $21$ &  $1$  \\ \hline 
 $10^{21}$ & $ 0.915$ & $173331$ &  $10070$ &  $903$ &  $48$ &  $2$  \\  \hline 
 $10^{22}$ & $ 0.713$ & $334737$ &  $17770$ &  $1617$ &  $91$ &  $3$  \\  \hline 
 $10^{23}$ & $ 0.552$ & $647265$ &  $30963$ &  $2930$ &  $174$ &  $5$  \\  \hline 
 $10^{24}$ & $ 0.426$ & $1253176$ &  $53564$ &  $5097$ &  $331$ &  $11$ \\  \hline 
\end{tabular}
\caption{Counts of imprimitive Carmichael numbers}
\label{table:imprimitive_CN}
\end{table}

\section{Related Problems}

\subsection{Future implementation}
There are two obvious implementation improvements if someone were to continue tabulating Carmichael numbers. 
First, the heuristically optimal algorithm could be tried and compared to our Algorithm \ref{alg:tabulation}.
Second, a switch away from GMP's arbitrary precision arithmetic to $128$-bit arithmetic could find significant performance increases.  
At the time of our implementation, we did not know of any library that could perform modular exponentiations faster than GMP in our range.  
However, we discovered Jeffrey Hurchalla's library \cite{128bit}; it seems to run about $40$\% faster for inputs in our range. 

\subsection{Absolute Lucas pseudoprimes}
A generalization of Carmichael numbers is to change the underlying divisibility sequence to Lucas sequences \cite{analog_carm}.  
In \cite{helm_webster}, it was shown that several of the themes of \cite{pinch15, sw_ants2} can be applied to tabulating these numbers and it seems natural that the results in this paper could be extended as well.  
In particular, we believe that these numbers can also be factored efficiently akin to Algorithm \ref{alg:car_in_p_cond} due to an identity comparable to the difference of squares factorization:
\[ U_{2^sn'}(A,B) = V_{n'}(A,B)\prod_{i = 0}^{s-1}U_{2^in'}(A,B),\]
where $U_k(A,B)$ and $V_k(A,B)$ take their usual meanings for Lucas sequences and $n'$ is the odd part of $n - \delta_n$ with $\delta_n$ being the Jacobi symbol of $n$ with respect to the discriminant of the sequence.

\subsection{Lehmer's totient problem}
Lehmer asked if there is a number $n$ such that $\phi(n) | (n-1)$.  
We believe that it should be possible to adopt the themes in this paper to design an algorithm to tabulate all examples less than $B$ in time $\widetilde{O}(B^{1/2})$ \cite{lehmer_totient_1, lehmer_totient_2}.

\subsection{Smallest Carmichael numbers with a fixed count of prime factors}
While we claim that our current work supersedes the prior work, that does not mean the prior work is irrelevant.  
Using the prime-by-prime approach in conjunction with composite completions (stopping when $P \lambda(P) > B$), undergraduates at Butler University, Aidan Johnson and Sylvia Webster, were able to find the smallest Carmichael numbers with exactly $36 \leq k \leq 39$ prime factors.  
The implementation used a heuristic guess for $B$ and updated to a smaller upper bound if a new candidate Carmichael number was found.
In this way we found the following numbers, extending Pinch's list by four entries:

\vspace{1em}
{\scriptsize
\begin{enumerate}[leftmargin=25pt]
\setcounter{enumi}{35}
\item  $8807647960173239406549096523101583812469752944862398843391133069112414401$
\item  $6434244746680627456690259889009091199393562391278723488104969729523529672001$
\item  $7058745493820052204846593069692928313507471037675454287017121977299708119200001$ 
\item  $2475018597201553460049849726869636604687455728815898384023241864617315492194337601$ 
\end{enumerate}
}

\section*{Acknowledgements}  We thank Richard Pinch, Carl Pomerance, Andrew Fiori, Ha Tran, and Jon Sorenson for helpful discussions that have made this paper better.  
Paul Pollack was especially helpful in clarifying some key points in Section \ref{sec_carinp}.  
Thank you to Drew Sutherland and the Simons Collaboration in Arithmetic Geometry, Number Theory, and Computation for providing computational support.
Thanks also to Frank Levinson for generous donations to Butler University that helped make the tabulation possible.

\bibliographystyle{amsplain}
\bibliography{manybases}

\end{document}